\documentclass[11pt]{amsart}

\usepackage{amsfonts}
\usepackage{amssymb}
\usepackage{hyperref}
\usepackage{stmaryrd}

\usepackage{tikz}
\usetikzlibrary{matrix,arrows,decorations.pathmorphing}

\theoremstyle{plain}
\newtheorem{thm}{Theorem}[section]
\newtheorem{lem}[thm]{Lemma}
\newtheorem{cor}[thm]{Corollary}
\newtheorem{prop}[thm]{Proposition}

\theoremstyle{definition}

\theoremstyle{remark}

\newcommand{\lemref}[1]{\hyperref[#1]{Lemma \ref*{#1}}}
\newcommand{\thmref}[1]{\hyperref[#1]{Theorem \ref*{#1}}}
\newcommand{\propref}[1]{\hyperref[#1]{Proposition \ref*{#1}}}
\newcommand{\corref}[1]{\hyperref[#1]{Corollary \ref*{#1}}}
\newcommand{\defref}[1]{\hyperref[#1]{Definition \ref*{#1}}}
\newcommand{\remref}[1]{\hyperref[#1]{Remark \ref*{#1}}}
\newcommand{\conjref}[1]{\hyperref[#1]{Conjecture \ref*{#1}}}

\newcommand{\Ker}{\mathrm{Ker}}

\def \Z {\mathbb{Z}}

\makeatletter
\newcommand*{\defeq}{\mathrel{\rlap{%
                     \raisebox{0.27ex}{$\m@th\cdot$}}%
                     \raisebox{-0.27ex}{$\m@th\cdot$}}%
                     =}

\numberwithin{equation}{section}
\makeatother

\makeatletter
\def\@setcopyright{}
\def\serieslogo@{}
\makeatother

\input xy
\xyoption{all} 
 
\begin{document}

\title{Virtual retraction and Howson's theorem in pro-$p$ groups}

\author{Mark Shusterman}
\address{Raymond and Beverly Sackler School of Mathematical Sciences, Tel-Aviv University, Tel-Aviv, Israel}
\email{markshus@mail.tau.ac.il}

\author{Pavel Zalesskii}
\address{Departamento de Matem\'atica,
~Universidade de Bras\'\i lia,
70910-900 Bras\'\i lia DF
Brazil}
\email{pz@mat.unb.br}

\date{}
\begin{abstract}

We show that for every finitely generated closed subgroup $K$ of a non-solvable Demushkin group $G$,
there exists an open subgroup $U$ of $G$ containing $K$, 
and a continuous homomorphism $\tau \colon U \to K$ satisfying $\tau(k) = k$ for every $k \in K$.
We prove that the intersection of a pair of finitely generated closed subgroups of a Demushkin group is finitely generated (giving an explicit bound on the number of generators).
Furthermore, we show that these properties of Demushkin groups are preserved under free pro-$p$ products,
and deduce that Howson's theorem holds for the Sylow subgroups of the absolute Galois group of a number field.
Finally, we confirm two conjectures of Ribes, thus classifying the finitely generated pro-$p$ M. Hall groups. 

\end{abstract}

\maketitle

\section{Introduction}

An immediate consequence of M. Hall's work \cite{Hall} is the virtual retraction of a free group on its finitely generated subgroups. 
In other words, given a finitely generated subgroup $K$ of a free group $F$, 
there exists a finite index subgroup $U \leq F$ containing $K$, 
and a subgroup $N \lhd U$ such that $KN = U$ and $K \cap N = \{1\}$ ($K$ is a retract of $U$).
The virtual retraction property has been established also for surface groups by Scott in \cite{Scott}, 
and is the subject of the work \cite{W} by Wilton on limit groups. 
Generalizing Hall's result on free groups, Haglund has shown in \cite{Hag} that quasi-convex subgroups of both right-angled Artin groups and hyperbolic right-angled Coxeter groups are virtual retracts.
Furthermore, virtual retraction plays an important role in the works \cite{BHW, HagWise} of Bergeron, Haglund and Wise,
and its significance in geometric group theory is highlighted by the work \cite{LR} of Long and Reid.
Moreover, as can be seen from \cite{KM, MV, MR, Rein},
retractions are closely related to equations over groups and to other questions of group theory.

In this work we continue the study of virtual retraction in the pro-$p$ setting.
This study has been initiated in \cite{Lub}, 
where Lubotzky established the virtual retraction property for finitely generated free pro-$p$ groups.
Lubotzky's result has been generalized in \cite{Rib} to arbitrary free pro-$p$ groups by Ribes.
In light of the aforementioned results,
virtual retraction has been proposed for pro-$p$ surface groups (pro-$p$ completions of fundamental groups of closed orientable surfaces) in \cite[Open Question 9.5.9]{RZ}.

Pro-$p$ surface groups belong to a very important class of pro-$p$ groups called Demushkin groups.
These are finitely generated one-relator pro-$p$ groups $G$ for which the cup product
\begin{equation}
\cup \colon H^1(G,\mathbb{F}_p) \times H^1(G,\mathbb{F}_p) \to H^2(G,\mathbb{F}_p)
\end{equation}
is non-degenerate.
Demushkin groups arise also as Galois groups of maximal $p$-extensions of $p$-adic fields,
and as maximal pro-$p$ quotients of etale fundamental groups of projective curves.
They were thoroughly studied by Demushkin, Serre, Labute and others (see \cite{Lab}),
and explicit forms for their relations were obtained.
Arithmetic and field theoretic aspects of Demushkin groups are treated in \cite{Ef2, Kon, Lab0, LLMS, MW0, MW, NSW, Ser, Wing}, 
and as can be seen from \cite{DL, Koc, KZ0, KZ1, SZ, Son}, 
contemporary group theorists continue the study of these groups.

We answer the aforementioned question on virtual retraction in pro-$p$ surface groups in the affirmative, 
thus generalizing the results of Lubotzky and Ribes, and providing the first family of non-free (pro-$p$) groups with the virtual retraction property.

\begin{thm} \label{FirstRes}

Let $G$ be a Demushkin group of rank greater than $2$,
and let $K$ be a closed topologically finitely generated subgroup of $G$.
Then there exists an open subgroup $U$ of $G$ containing $K$ such that $K$ is a continuous retract of $U$.

\end{thm}

The theorem says that every finitely generated subgroup of a Demushkin group (of rank $> 2$) is a semi-direct factor of some finite-index subgroup.
Our assumption that the rank of $G$ exceeds $2$ is necessary, 
as shown by $K \defeq \mathbb{Z}_3$ in $G \defeq K \rtimes \mathbb{Z}_3$ with a nontrivial action.

One of the obstructions to virtual retraction in Demushkin groups is the existence of torsion in the abelianization of open subgroups. 
It is therefore a consequence of \thmref{FirstRes} that the image of $K$ in the abelianization of $U$ is torsion-free. 
In the course of the proof of \thmref{FirstRes} we develop a theory of symplectic forms over local rings,
and give a simpler proof for the virtual retraction property in free pro-$p$ groups 
(more generally, we establish the virtual retraction of projective profinite groups on certain finitely generated subgroups).
One of the main tools in the proof of \thmref{FirstRes} is the dualizing module,
and it is advisable to read the proof having in mind the case where $G$ is the Galois group of the maximal $p$-extensions of $\mathbb{Q}_p(\zeta_p)$.
In this case, many of our considerations acquire an arithmetic flavor.
Apart from the general theory of Demushkin groups given in \cite{Lab},
our proof uses and generalizes ideas from \cite{Son}. 
As a consequence of \thmref{FirstRes}, 
we obtain a pro-$p$ analogue of a theorem by Karrass and Solitar, 
variants of which are studied in \cite{Arzh, Jar, Steinb}.

\begin{cor} \label{ImmCor}

Let $G$ be a Demushkin group of rank greater than $2$,
and let $K$ be a closed topologically finitely generated subgroup of infinite index in $G$.
Then there exists a nontrivial closed normal subgroup $N$ of $G$ that intersects $K$ trivially.

\end{cor}

We further enrich the family of pro-$p$ groups possessing the virtual retraction property.

\begin{thm} \label{PresVRThm}

The virtual retraction property is preserved under free pro-$p$ products.

\end{thm}

We also give an example of a (non-free) pro-$p$ group that has the virtual retraction property, but is not finitely generated.

\begin{cor} \label{adicSylowCor}

A $p$-Sylow subgroup of the absolute Galois group of $\mathbb{Q}_p$ has the virtual retraction property.

\end{cor} 

This is proved using \thmref{FirstRes} and the structure theorem of such Sylow subgroups from \cite{Lab0}.

As shown by Lubotzky in \cite{Lub} and Ribes in \cite{Rib},
a property stronger than virtual retraction holds in free pro-$p$ groups.
Namely, every finitely generated subgroup is virtually a free factor.
Groups with the latter property are called M. Hall groups.
Motivated by the discrete analog considered in \cite{Bog1, Bog2, BB, Tret, Tret2} and in other works, 
Ribes has shown in \cite{Rib} that being M. Hall is preserved under forming finitely generated free pro-$p$ products,
and conjectured in \cite[Conjecture 5.3]{Rib} that his finite generation assumption is superfluous.
Herein, we confirm his conjecture, thus obtaining a pro-$p$ analog of \cite[Theorem 1.1]{Bur00}.

\begin{thm} \label{PresMHThm}

The M. Hall property is preserved under free pro-$p$ products.

\end{thm}

In \cite[Conjecture 5.2]{Rib}, Ribes suggests a converse to \thmref{PresMHThm},
saying that all the finitely generated pro-$p$ M. Hall groups are given by free pro-$p$ products of the `basic' M. Hall groups.
We prove this conjecture as well. 

\begin{thm} \label{RibThm}

Every finitely generated pro-$p$ M. Hall group is a free pro-$p$ product of finite $p$-groups and procyclic pro-$p$ groups.

\end{thm} 

Our proof relies on results and ideas from the recent work \cite{WZ} on the pro-$p$ analog of Stallings' decomposition theorem. In particular, we use \cite[Corollary B]{WZ}, saying that a finitely generated torsion-free pro-$p$ group which is virtually a free pro-$p$ product, is itself a free pro-$p$ product.
It should be noted that despite the attention devoted to the M. Hall property,
there is no result analogous to \thmref{RibThm} for discrete groups.

Let us now shift our attention to Howson's theorem from \cite{Hows},
saying that the intersection of a pair of finitely generated subgroups of a free group is finitely generated.
This theorem has been extended by Greenberg to surface groups in \cite{Green}, by Dahmani to limit groups in \cite{Da}, 
and to many other discrete groups (see for instance \cite{Baum, Bur, Bur2, Hem, Kap, KS, Para, Som}).
The pro-$p$ sibling of Howson's theorem has been obtained for free pro-$p$ groups by Lubotzky in \cite{Lub}.
More generally, we have the following.

\begin{prop} \label{HimpHprop}

Let $G$ be an M. Hall group, and let $H,K$ be finitely generated subgroups of $G$.
Then $H \cap K$ is finitely generated.

\end{prop}

\begin{proof}

There exists a finite index subgroup $U$ of $G$ containing $H$ as a free factor.
By the Kurosh subgroup theorem, $K \cap U \cap H$ is a free factor of $K \cap U$, so it is finitely generated.
Since $K \cap U \cap H$ is of finite index in $K \cap H$ we conclude that the latter is finitely generated, as required.
\end{proof}

\propref{HimpHprop} is valid for both discrete and pro-$p$ groups,
thus enlarging the family of groups satisfying Howson's theorem.
The proof of \propref{HimpHprop} is a substantially simplified version of the original arguments by Howson and Lubotzky.
Yet, this method for proving Howson's theorem fails for Demushkin groups (since these are not M. Hall as shown by \thmref{RibThm}),
so we obtain the result using a different method.

\begin{thm} \label{SecRes}

Let $A,B$ be closed topologically finitely generated subgroups of a Demushkin group.
Then $A \cap B$ is topologically finitely generated.

\end{thm}

This is the first extension of Howson's theorem to pro-$p$ groups that are not free.
Our proof of \thmref{SecRes}, which relies on ideas from \cite{Sh}, does not merely show that the intersection is finitely generated,
but also gives a bound on the number of generators of $A \cap B$ that depends only on the numbers of generators of $A$ and $B$.
Such bounds are at the focus of the Hanna Neumann conjecture,
that was studied in a plethora of works (see for example \cite{Bur0, Di, DiF, E, Ger, Nh, Ni, T, T2})
until it was finally resolved by Friedman and Mineyev in \cite{M, M2} and \cite{F} respectively.
Analogs of the conjecture for other discrete groups are considered in several works such as \cite{Bur3, DiI, I, I2, Som0, Som1, Z}.
Giving any bound on the number of generators of the intersection in the free pro-$p$ case was an open problem for more than 30 years (Lubotzky's argument does not provide a bound), until a breakthrough was made by Jaikin-Zapirain in \cite{J}, leading also to a new proof of the original Hanna Neumann conjecture.
In \eqref{HowsGoalEq} we obtain such a bound for Demushkin groups, thus establishing \thmref{SecRes}.

Our arguments from the proof of \thmref{SecRes} generalize to other pro-$p$ groups.
Furthermore, they can be used to reprove Howson's theorem for surface groups, 
and to obtain bounds on the number of generators of the intersection that are better than those previously known (see  \cite{Bur3, Som0, Som1}) in certain special cases (see Section \ref{HowsSec}).
 
We also establish the pro-$p$ analogue of Baumslag's main result from \cite{Baum}.

\begin{thm} \label{BaumThm}

If Howson's theorem holds for pro-$p$ groups $G,H$ then it also holds for their free pro-$p$ product $G \amalg H$.

\end{thm}

Using Efrat's result from \cite{Ef} that describes finitely generated pro-$p$ subgroups of absolute Galois groups of global fields (of characteristic $\neq p$) as free pro-$p$ products of subgroups of Demushkin groups, 
we immediately arrive at the following corollary of \thmref{SecRes} and \thmref{BaumThm}.

\begin{cor}

A $p$-Sylow subgroup of the absolute Galois group of a global field of characteristic different from $p$ satisfies Howson's theorem.

\end{cor}

For more studies of the Sylow subgroups of absolute Galois groups (motivated by a question of Serre) see for instance \cite{BJN, BLMS, Lab0}.

\section{Virtual retraction in Demushkin groups}

Some parts of this section are based on arguments from \cite{Son}.
In order to make these arguments suitable for our needs,
they are presented here in a detailed and generalized form. 

\subsection{Symplectic forms}

Let $R$ be a (commutative unital) local ring, 
let $\mathfrak{m}$ be its unique maximal ideal,
let $\kappa = R/\mathfrak{m}$ be its residue field,
and let $M$ be a nontrivial finitely generated free $R$-module.
A bilinear form $\omega \colon M \times M \to R$ is called symplectic if it is skew-symmetric and nondegenerate.
That is,
\begin{itemize}

\item For all $a,b \in M$ we have $\omega(a,b) = -\omega(b,a)$.

\item For some (equivalently, every) basis $c_1, \dots, c_n$ of $M$ over $R$ we have $\det\big(\omega(c_i,c_j)\big) \in R^*$.

\end{itemize}
Evidently, $\omega$ induces a symplectic bilinear form on the $\kappa$-vector space
\begin{equation}
M_\kappa \defeq \kappa \otimes_R M \cong M/\mathfrak{m}M.
\end{equation}
We assume throughout that the rank $n$ of $M$ over $R$ (which equals $\dim_{\kappa} M_\kappa$) is even and positive.
For $t \defeq \frac{n}{2}$, 
we say that a sequence $(a_1, b_1, \dots, a_t, b_t)$ of elements of $M$ is a symplectic basis for $M$,
if it is a basis of $M$ over $R$ and if 
\begin{equation} \label{DefSymBasEq}
\begin{split}
&\forall \ 1 \leq i \leq t \quad \omega(a_i,b_i) = 1, \\
&\forall \ 1 \leq i \neq j \leq t \quad \omega(a_i, b_j) = \omega(a_i, a_j) = \omega(b_i, b_j) = 0.
\end{split}
\end{equation}
We write $S_1 \frown S_2$ for the concatenation of the sequences $S_1$ and $S_2$.
The following observation allows us to construct symplectic bases.

\begin{prop} \label{ObsBaseProp}
For every even-dimensional nondegenerate subspace $L$ of $M_\kappa$,
the orthogonal complement $L^\bot$ of $L$ in $M_\kappa$ is nondegenerate,
and satisfies $L \oplus L^\bot = M_\kappa$.
Furthermore, if $B,C$ are symplectic bases for $L,L^\bot$ respectively,
then $B \frown C$ is a symplectic basis for $M_\kappa$.
\end{prop}

\begin{proof}
The nondegeneracy of $\omega$ implies that $\dim_\kappa L^\bot = n - \dim_\kappa L$,
and the nondegeneracy of $L$ implies that $L \cap L^\bot = 0$, 
so we infer that $L \oplus L^\bot = M_\kappa$ and thus that $B \frown C$ is a symplectic basis for $M_\kappa$.
The nondegeneracy of $L^\bot$ follows from the fact that $\det(\omega) = \det(\omega|_L)\det(\omega|_{L^\bot})$.
\end{proof}

In order to be able to apply \propref{ObsBaseProp} we need the following.

\begin{prop} \label{TwoProp}
$M_\kappa$ has a $2$-dimensional nondegenerate subspace.
\end{prop}

\begin{proof}
Suppose that there exists a $0 \neq u \in M_\kappa$ with $\omega(u,u) = 0$. 
Nondegeneracy gives us a $v \in M_\kappa$ with $\omega(u,v) \neq 0$ so we see that $\mathrm{Span}_\kappa \{u, v\}$ is nondegenerate.
Otherwise, 
taking any $y,z \in M_\kappa \setminus \{0\}$ with $\omega(y,z) = 0$ we get that $\mathrm{Span}_\kappa \{y, z\}$ is nondegenerate as $\omega(y,y), \omega(z,z) \neq 0$.
\end{proof}

The following proposition bares resemblance to \cite[Theorem 8.1]{Lan}.

\begin{prop} \label{SympBaseExProp}

There exists a symplectic basis for the vector space $M_\kappa$.

\end{prop}

\begin{proof}
We induct on $n = \dim_\kappa M_\kappa$, so assume first that $M_\kappa = \mathrm{Span}_\kappa \{u,v\}$.
Nondegeneracy implies that $U \defeq \mathrm{Span}_\kappa \{u\}$ and $\ U^\bot$ are proper subspaces of $M_\kappa$,
so there is a $z \in M_\kappa$ avoiding both.
Hence $(u,\omega(u,z)^{-1}z)$ is a symplectic basis.
Assume now that $n > 2$, and use \propref{TwoProp} to pick a nondegenerate $2$-dimensional subspace 
$L$ of $M_\kappa$.
By \propref{ObsBaseProp}, $L^\bot$ is nondegenerate as well,
so by induction, both $L$ and $L^\bot$ have symplectic bases.
Our basis is constructed by invoking \propref{ObsBaseProp}.
\end{proof}

A submodule $N$ of $M$ is called isotropic if $\omega(a,b) = 0$ for all $a,b \in N$.
Nondegeneracy alone gives the following.

\begin{lem} \label{HalfLem}
The dimension of any isotropic subspace of $M_\kappa$ is at most $t$.
\end{lem}

\begin{proof}
Suppose that $c_1, \dots, c_{t+1}$ is a basis of an isotropic subspace of $M_\kappa$.
Completing it to a basis $c_1, \dots, c_n$ of $M_\kappa$, 
we see that the matrix $\omega(c_i,c_j)$ has a $(t+1)\times(t+1)$ block of zeros on the upper-left.
Hence, every generalized diagonal contains a zero, 
so the determinant vanishes.
\end{proof}

\begin{prop} \label{CompProp}

Let $N$ be an isotropic subspace of $M_\kappa$, 
and let $0 \neq b_1 \in N$.
Then there exists a symplectic basis $(a_1,b_1, \dots, a_t, b_t)$ of $M_\kappa$ such that
\begin{equation}
N \subseteq \mathrm{Span}_\kappa \{b_1, \dots, b_t\}.
\end{equation}

\end{prop}

\begin{proof}

Let $b_1, \dots, b_s$ be a basis for $N$ over $\kappa$.
By \lemref{HalfLem}, we know that $s \leq t$.
We construct the elements $a_1, \dots, a_s \in M_\kappa$ inductively,
so suppose that for some $r \leq s$ we have already selected $a_i \in M_\kappa$ for all $i < r$ such that 
\begin{equation} \label{IndEq}
\{a_i\}_{i < r} \cup \{b_1, \dots, b_s\}
\end{equation}
is an independent set, and
\begin{equation} \label{OrthEq}
\begin{split}
&\forall \ 1 \leq i < r \quad \omega(a_i,b_i) = 1, \ \forall \ 1 \leq j \neq i < r \quad \omega(a_i, a_j) = 0, \\
&\forall \ 1 \leq i < r, \ 1 \leq j \leq s, \ i \neq j, \quad \omega(a_i, b_j) =  0.
\end{split}
\end{equation}
Since \eqref{IndEq} is independent,
nondegeneracy supplies an $a_r \in M_\kappa$ such that
\begin{equation} \label{RCondEq}
\forall \ 1 \leq i < r, \ 1 \leq r \neq j \leq s \quad \omega(a_r,b_j) = \omega(a_r, a_i) = 0, \ \omega(a_r, b_r)=1.
\end{equation} 
Note that by \eqref{OrthEq}, the set \eqref{IndEq} is orthogonal to $b_r$,
so $a_r$ is not in this set's span since $\omega(a_r,b_r) = 1$ by \eqref{RCondEq}.
We conclude that 
\begin{equation}
\{a_i\}_{i < r+1} \cup \{b_1, \dots, b_s\}
\end{equation}
is independent,
and that \eqref{OrthEq} holds for $r+1$.
Induction implies that
\begin{equation} \label{FBasisEq}
(a_1,b_1, \dots, a_s,b_s)
\end{equation}
is a symplectic basis for its span $L$.
Using this basis one immediately verifies that $\det(\omega|_L) = 1$ so $L$ is nondegenerate.
By \propref{ObsBaseProp}, $L^\bot$ is nondegenerate as well, 
so by \propref{SympBaseExProp} it has a symplectic basis
\begin{equation} \label{SBasisEq}
(a_{s+1},b_{s+1}, \dots, a_t,b_t).
\end{equation}
Appending \eqref{SBasisEq} to \eqref{FBasisEq} gives the desired basis (by \propref{ObsBaseProp}).
\end{proof}

We have a reduction mod $\mathfrak{m}$ homomorphism of $R$-modules $\rho \colon M \to M_\kappa$ induced by the homomorphism $\rho \colon R \to \kappa$ (an abuse of notation, of course).
By definition, $\rho$ respects the symplectic forms. That is,
\begin{equation} \label{RespectEq}
\forall a,b \in M \quad \omega\big(\rho(a),\rho(b)\big) = \rho\big(\omega(a,b)\big).
\end{equation}

\begin{prop} \label{LiftProp}

Let $B_1 \in M$, set $b_1 \defeq \rho(B_1)$, and suppose that 
\begin{equation}
D = (a_1,b_1, \dots, a_t,b_t)
\end{equation}
is a symplectic basis of $M_\kappa$.
Then $B_1$ can be completed to a symplectic basis 
\begin{equation}
C = (A_1,B_1, \dots, A_t,B_t)
\end{equation}
of $M$ over $R$ such that $\rho(C) = D$.

\end{prop}

\begin{proof}
Take any basis $C = (A_1,B_1, \dots, A_t,B_t)$ of $M$ over $R$ such that $\rho(C) = D$.
Since $D$ is symplectic, \eqref{DefSymBasEq} tells us that
\begin{equation} \label{ModEq}
\begin{split}
&\forall \ 1 \leq i \leq t \quad \omega(A_i,B_i) \equiv 1 \mod \mathfrak{m}, \\
&\forall \ 1 \leq i \neq j \leq t \quad 
\omega(A_i, B_j) \equiv \omega(A_i, A_j) \equiv \omega(B_i, B_j) \equiv 0 \mod \mathfrak{m}.
\end{split}
\end{equation}

Let us now `symplectify' $C$ without changing the fact that it is a basis of $M$ that reduces to $D$ mod $\mathfrak{m}$.
Since $R$ is a local ring, and by \eqref{ModEq} we have $\omega(A_1,B_1) \notin \mathfrak{m}$, we know that $\omega(A_1,B_1) \in R^*$
so by replacing
\begin{equation}
A_i \leftarrow A_i - \frac{\omega(A_i,B_1)}{\omega(A_1,B_1)}A_1, \quad 
B_i \leftarrow B_i - \frac{\omega(B_i,B_1)}{\omega(A_1,B_1)}A_1 \quad (i > 1)
\end{equation}
we assure that $\omega(A_i,B_1) = \omega(B_i,B_1) = 0$.
We will now also make sure that $\omega(A_1,A_i) = \omega(A_1,B_i) = 0$ for $i > 1$.
Since $D$ is symplectic, $\omega$ is represented in this basis by a block matrix,
with a $2 \times 2$ block in the upper-left and an $(n-2)\times(n-2)$ block $T$ in the lower-right (recall that $n=2t$).
Nodegeneracy implies that $T$ is invertible,
so the lower-right $(n-2)\times(n-2)$ submatrix $\Lambda$ of the matrix representing $\omega$ in the basis $C$ is also invertible, as it reduces to $T$ mod $\mathfrak{m}$.
We can thus find a $v = (\alpha_2, \beta_2, \dots, \alpha_t, \beta_t) \in \mathfrak{m}^{n-2}$ such that
\begin{equation} \label{LinEq}
v\Lambda = \big(\omega(A_1,A_2), \omega(A_1,B_2), \dots, \omega(A_1,A_t), \omega(A_1,B_t)\big).
\end{equation}
Upon replacing
\begin{equation}
A_1 \leftarrow A_1 - (\alpha_2A_2 + \beta_2B_2 + \dots + \alpha_tA_t + \beta_tB_t)
\end{equation}
one verifies that \eqref{LinEq} guarantees that $\omega(A_1,A_i) = \omega(A_1,B_i) = 0$ for $i > 1$.
After multiplying $A_1$ by $\omega(A_1,B_1)^{-1}$ we also have $\omega(A_1,B_1) = 1$.

Since $L \defeq \mathrm{Span}_R\{A_2,B_2, \dots, A_t,B_t\}$ is orthogonal to $\mathrm{Span}_R\{A_1,B_1\}$,
it follows that $L$ is nondegenerate,
so by induction, we may assume that $(A_2,B_2, \dots, A_t,B_t)$ is a symplectic basis that reduces to $(a_2,b_2, \dots, a_t, b_t)$.
Therefore $(A_1,B_1, \dots, A_t,B_t)$ is the required symplectic basis.
\end{proof}

An immediate consequence of \propref{SympBaseExProp} and \propref{LiftProp} is the following.

\begin{cor}
Any invertible skew-symmetric $2t \times 2t$ matrix over a local ring is congruent to a block matrix with $t$ blocks of the form
\begin{equation}
\left(\begin{array}{cc} * & 1\\ -1 & * \end{array}\right)
\end{equation}
on the main diagonal, and zeros at all the other entries.
\end{cor}

The material in this section is standard for those concerned only with fields of characteristic different from $2$, 
or with forms that are alternating (a condition stronger than skew-symmetry).
However, we could not find a unified treatment of (not necessarily alternating) forms over rings with arbitrary residue characteristic in the literature,
so we have included it in this section for the reader's convenience.

\subsection{Profinite groups}

\subsubsection{Generalities}

We will be dealing with profinite groups, so all notions should be interpreted in the topological sense.
For instance, subgroups are closed, and homomorphisms are continuous.

\begin{lem} \label{EmbProbLem}

A finitely generated subgroup $K$ of a profinite group $G$ is a retract if and only if there exists a homomorphism 
$\lambda \colon G \to K/\Phi(K)$ extending the Frattini map $\varphi \colon K \to K/\Phi(K)$,
such that the embedding problem
\begin{equation*}
\begin{tikzpicture}[scale=1.5]
\node (B) at (1.4,1) {$G$};
\node (C) at (0,0) {$K$};
\node (D) at (1.4,0) {$K/\Phi(K)$};
\path[->,font=\scriptsize,>=angle 90]
(B) edge node[right]{$\lambda$} (D)
(C) edge node[above]{$\varphi$} (D);
\end{tikzpicture}
\end{equation*}
is weakly solvable (that is, there exists a homomorphism $\mu \colon G \to K$ that makes the diagram commutative).

\end{lem}

\begin{proof}

If $\tau \colon G \to K$ is a retraction, then by setting $\lambda \defeq \varphi \circ \tau$ we get an extension of $\varphi$ for which $\tau$ is a weak solution to our embedding problem. 

For the other direction, suppose that $\mu \colon G \to K$ is a weak solution to our embedding problem 
(that is, $\lambda = \varphi \circ \mu$). We have
\begin{equation}
K/\Phi(K) = \varphi(K) = \lambda(K) = \varphi\big(\mu(K)\big)
\end{equation}
so we conclude that $\mu(K)\Phi(K) = K$.
From \cite[Corollary 2.8.5]{RZ}, we get that $\mu(K) = K$.
Taking $N \defeq \Ker( \mu)$ we see that $KN = G$ and since $K$ is finitely generated,
\cite[Proposition 2.5.2]{RZ} tells us that $K \cap N = \{1\}$.
\end{proof}

Recall that a finite group $S$ is called supersolvable if all its chief factors are cyclic 
(for example, a nilpotent group).
Accordingly, a profinite group $S$ is said to be prosupersolvable, 
if it is an inverse limit of finite supersolvable groups. 
Prosupersolvable groups are studied, for instance, in \cite{AS, OR, RZ, Sh2, Sm}.
Using \lemref{EmbProbLem}, we generalize the virtual retraction property of free pro-$p$ groups that follows from \cite[Theorem 9.1.19]{RZ} (see also \cite{Lub, Rib}).

\begin{cor} \label{SuperVRPCor}

Let $G$ be a projective profinite group, and let $K$ be a finitely generated prosupersolvable subgroup whose order is divisible by only finitely many different primes.
Then $K$ is virtually a retract of $G$.

\end{cor}

\begin{proof}

\cite[Proposition 2.8.11]{RZ} implies that $\Phi(K) \leq_o K$ so by \cite[Lemma 1.2.5 (c)]{FJ}, 
the Frattini quotient map $\varphi \colon K \to K/\Phi(K)$ extends to a homomorphism 
$\lambda \colon U \to K/\Phi(K)$ where $U$ is some open subgroup of $G$ containing $K$.
By \cite[Theorem D.4.1]{RZ}, $U$ is projective, so the embedding problem
\begin{equation}
\begin{tikzpicture}[scale=1.5]
\node (B) at (1.4,1) {$U$};
\node (C) at (0,0) {$K$};
\node (D) at (1.4,0) {$K/\Phi(K)$};
\path[->,font=\scriptsize,>=angle 90]
(B) edge node[right]{$\lambda$} (D)
(C) edge node[above]{$\varphi$} (D);
\end{tikzpicture}
\end{equation}
is weakly solvable. By \lemref{EmbProbLem}, $K$ is a retract of $U$.
\end{proof}

Subgroups $K$ to which \corref{SuperVRPCor} is applicable (such as (free) pro-$p$ groups) are provided by the following claim, whose proof is in \cite{Sm}.

\begin{prop}

Let $D$ be a finite set of primes, let $\mathcal{C}$ be the family of finite supersolvable groups whose order is not divisible by primes not in $D$, and let $K$ be a free pro-$\mathcal{C}$ group.
Then $K$ is projective.

\end{prop}

%

\subsubsection{Demushkin groups}

We will use throughout the theory of Demushkin groups as presented in \cite{Lab}.
It is advisable to consult \cite{NSW} and \cite{Ser} as well.

Let $p$ be a prime number, let $n = 2t$ be an even integer exceeding $2$, 
let $G$ be a pro-$p$ Demushkin group with $d(G) = n$, and set $q \defeq q(G)$ (this is $0$ if $G^{\mathrm{ab}}$ is torsion-free, and otherwise equals the order of its torsion subgroup).
We say that we are in the Galois case, 
if $G$ is the Galois group of the maximal $p$-extension of a $p$-adic field $E$.
In this case, $q$ is the highest power of $p$ for which $E$ contains a primitive $q$-th root of unity $\zeta_q$.

Let $\mathbb{Z}_p$ be the ring of $p$-adic integers, 
and let $R$ be the local ring $\mathbb{Z}_p/q\mathbb{Z}_p$ with residue field $\kappa = \mathbb{F}_p$.
We have a presentation for $G$
\begin{equation}
1 \longrightarrow L \longrightarrow F \stackrel{e}{\longrightarrow} G \longrightarrow 1
\end{equation}
where $F$ is a free pro-$p$ group on $n$ generators,
and $L$ is generated (as a normal subgroup of $F$) by some $r \in F^q[F,F]$.
The $q$-central series of $F$ is
\begin{equation}
F_1 \defeq F, \quad F_{k+1} \defeq F_k^q[F_k,F] \quad (k \geq 1).
\end{equation}
As $r \in F_2$, when $q \neq 0$ we can find an $s \in F$ with
\begin{equation}
s^q \equiv r \mod [F,F]
\end{equation}
and set $\sigma \defeq e(s)$.
The values $\eta(\sigma) \in R$ for $\eta \in M \defeq H^1(G,R)$ are independent of our choice of $s$.
Cup product induces a symplectic form on the $R$-module $M$,
so (when $q \neq 0$) there exists a $\chi_\sigma \in M$ such that
\begin{equation} \label{sigmacupEq}
\forall \ \eta \in M \quad \eta \cup \chi_\sigma = \eta(\sigma).
\end{equation}
In the Galois case, cup product corresponds to the Hilbert symbol. 
\cite[Corollary to Proposition 3]{Lab} tells us that (when $q \neq 0$)
\begin{equation} \label{selfcupEq}
\forall \ \eta \in M \quad \eta \cup \eta = \binom{q}{2} \eta(\sigma)
\end{equation}
so if we take a symplectic basis $\chi_1, \chi_2 \dots, \chi_n$ for $M$ over $R$ with
\begin{equation} \label{ActOnSigEq}
\chi_i(\sigma) = \delta_{1,i} \quad (1 \leq i \leq n)
\end{equation}
whose existence is guaranteed by \cite[Proposition 4 (2)]{Lab},
then using \eqref{sigmacupEq} and \eqref{selfcupEq} we can show that the functionals 
\begin{equation} 
\eta \mapsto \eta \cup \chi_2, \quad \eta \mapsto \eta \cup \chi_\sigma
\end{equation}
coincide on the basis. Consequently, nondegeneracy implies that $\chi_2 = \chi_\sigma$ so
\begin{equation} \label{SigmaIsotEq}
\chi_\sigma(\sigma) \stackrel{\ref{sigmacupEq}}{=} \chi_\sigma \cup \chi_\sigma = \chi_2 \cup \chi_\sigma
\stackrel{\ref{sigmacupEq}}{=} \chi_2(\sigma) \stackrel{\ref{ActOnSigEq}}{=} 0.
\end{equation}

Cup product also induces a symplectic form on the vector space 
\begin{equation}
H^1(G, \kappa) \cong M_\kappa.
\end{equation}
Identifying these, 
we get a reduction homomorphism $\rho \colon M \to M_\kappa$ respecting the symplectic forms (see \eqref{RespectEq}).

Associated to $G$ (or to its dualizing module) is a homomorphism 
\begin{equation} \label{chidefeq}
\chi \colon G \to \mathbb{U}_p^{1} \defeq \{a \in \mathbb{Z}_p^* \ | \ a \equiv 1 \ (p)\}
\end{equation}
for which $q$ is the highest power of $p$ with 
\begin{equation}
\mathrm{Im}(\chi) \subseteq  1 + q\mathbb{Z}_p
\end{equation}
where we think of $q = 0$ as $p^\infty$.
Thus, when $q \neq 0$, reduction mod $pq$ maps $\mathrm{Im}(\chi)$ into the subgroup of $(\mathbb{Z}/pq\mathbb{Z})^*$ given by
\begin{equation}
\{b \in (\mathbb{Z}/pq\mathbb{Z})^* : b \equiv 1 \ (q)\} = \{1 + tq : 0 \leq t \leq p-1\}
\end{equation}
which is isomorphic to $(\kappa,+)$ by $1 + tq \mapsto t$.
Hence, we get a homomorphism
\begin{equation}
\tau \colon \mathrm{Im}(\chi) \to (\kappa,+).
\end{equation}
In the Galois case, $\chi$ maps each $g \in G$ to the automorphism that $g$ induces on the multiplicative group of roots of unity in $\bar E$.
  
Rephrasing \cite[Lemma 6]{Son} we get a statement,
the proof of which can also be found in \cite[Corollary 3.9.17]{NSW} and in the corollary to \cite[Proposition 4]{Lab}.

\begin{prop} \label{ApproxProp}

Let $(A_1, B_1, \dots, A_t,B_t)$ be a symplectic basis of $M$, 
and let $\alpha_1, \beta_1, \dots, \alpha_t, \beta_t$ be dual generators of $G$.
Let $w_1, z_1, \dots, w_t, z_t$ be a basis of $F$ with $e(w_i) = \alpha_i, \ e(z_i) = \beta_i$ for $1 \leq i \leq t$.
If $q \neq 0$, suppose in addition that $B_1 = \chi_\sigma$. Then 
\begin{equation}
r \equiv w_1^q[w_1, z_1] \cdots [w_t, z_t] \mod F_3.
\end{equation}

\end{prop}

Following the argument in the proof of \cite[Theorem 3]{Lab} reproduced also in the proofs of \cite[Theorem 7]{Son} and \cite[Theorem 3.9.11, Lemma 3.9.18]{NSW} we extract the following assertion.

\begin{prop} \label{ExactProp}

Let $w_1, z_1, \dots, w_t, z_t$ be a basis of $F$ for which
\begin{equation}
r \equiv w_1^q[w_1, z_1] \cdots [w_t, z_t] \mod F_3.
\end{equation}
Then there exists a basis $x_1, y_1, \dots, x_t, y_t$ of $F$ and $\gamma, \delta \in \mathbb{Z}_p$ such that 
\begin{equation}
\begin{split}
&r = x_1^\gamma[x_1, y_1]x_2^\delta[x_2,y_2][x_3,y_3] \cdots [x_t, y_t], \\
&\forall \ 1 \leq i \leq t \quad x_i \equiv w_i, \ y_i \equiv z_i \mod F_2.
\end{split}
\end{equation}

\end{prop}

The symplectic basis needed for \propref{ApproxProp} is provided by a combination of \eqref{SigmaIsotEq}, \propref{CompProp}, and \propref{LiftProp}
(alternatively, one can take the basis $\chi_1, \dots, \chi_n$ from \cite[Proposition 4 (2)]{Lab} discussed above).
Then, \propref{ApproxProp} and \propref{ExactProp} provide a basis for $F$ such that
\begin{equation} \label{ActEq}
\begin{split}
&\forall \ 1 \leq i \leq t \quad \chi_\sigma\big(e(x_i)\big) = \chi_\sigma\big(e(w_i)\big) = \chi_\sigma(\alpha_i) = B_1(\alpha_i) = 0, \\ 
&\forall \ 1 \leq i \leq t \quad \chi_\sigma\big(e(y_i)\big) = \chi_\sigma\big(e(z_i)\big) = \chi_\sigma(\beta_i) = B_1(\beta_i) = \delta_{1,i},
\end{split}
\end{equation}
so we have an explicit description of $\chi_\sigma$ (as opposed to \eqref{sigmacupEq}).
An explicit description of $\chi$ from \eqref{chidefeq} is given in the proof of \cite[Theorem 4]{Lab}.  
It follows from these descriptions that
\begin{equation} \label{DescrEq}
\rho(\chi_\sigma) = \tau \circ \chi.
\end{equation}
In the Galois case,
this means that the fixed field of $\mathrm{Ker}\big(\rho(\chi_\sigma)\big)$ is $E(\zeta_{pq})$.

\begin{cor} \label{IsoCor}

For every isotropic subspace $N$ of $M_\kappa$ that contains $\rho(\chi_\sigma)$ (if $q \neq 0$),
there exists a basis $x_1, y_1, \dots, x_t, y_t$ of $F$ and $\gamma, \delta \in \mathbb{Z}_p$ such that
\begin{equation}
\begin{split}
&r = x_1^\gamma[x_1, y_1]x_2^\delta[x_2,y_2][x_3,y_3] \cdots [x_t, y_t], \\
&\forall \psi \in N \ \forall \ 1 \leq i \leq t \quad \psi\big(e(x_i)\big) = 0.
\end{split}
\end{equation}

\end{cor}

\begin{proof}

By \eqref{SigmaIsotEq} and \propref{CompProp}, 
there exists a symplectic basis 
\begin{equation} \label{FirstBasEq}
(a_1,b_1, \dots, a_t, b_t)
\end{equation}
of $M_\kappa$ such that
\begin{equation} \label{NSubEq}
b_1 = \rho(\chi_\sigma) \ (\mathrm{if} \ q \neq 0), \quad N \subseteq \mathrm{Span}_\kappa \{b_1, \dots, b_t\}.
\end{equation}
Using \propref{LiftProp}, 
we lift \eqref{FirstBasEq} to a symplectic basis
\begin{equation} \label{SecBasEq}
(A_1, B_1, \dots, A_t, B_t)
\end{equation}
of $M$ with $B_1 = \chi_\sigma$ in case that $q \neq 0$.
Let $\alpha_1, \beta_1, \dots, \alpha_t, \beta_t$ be dual generators of $G$.
This means that
\begin{equation} \label{PresEq}
\forall \ 1 \leq i,j \leq t \quad B_j(\alpha_i) = A_j(\beta_i) = 0, \ A_j(\alpha_i) = B_j(\beta_i) = \delta_{i,j}.
\end{equation}
Let $w_1, z_1, \dots, w_t, z_t$ be a basis of $F$ such that
\begin{equation} \label{AWEq}
\forall \ 1 \leq i \leq t \quad e(w_i) = \alpha_i, \ e(z_i) = \beta_i.
\end{equation}
By \propref{ApproxProp}, we have
\begin{equation}
r \equiv w_1^q[w_1, z_1] \cdots [w_t, z_t] \mod F_3.
\end{equation} 
By \propref{ExactProp}, there exists a basis $x_1, y_1, \dots, x_t, y_t$ of $F$ such that 
\begin{equation} \label{EquivXWEq}
r = x_1^\gamma[x_1, y_1]x_2^\delta \cdots [x_t, y_t], \quad \forall \ 1 \leq i \leq t \quad x_i \equiv w_i \mod F_2 .
\end{equation}
Consequently, for all $1 \leq i, j \leq t$ we have
\begin{equation}
b_j\big(e(x_i)\big) \stackrel{\ref{EquivXWEq}}{=} 
b_j\big(e(w_i)\big) \stackrel{\ref{AWEq}}{=} 
b_j(\alpha_i) \stackrel{\ref{SecBasEq}}{=} 
\rho\big(B_j(\alpha_i)\big) \stackrel{\ref{PresEq}}{=} 0
\end{equation}
so we are done in view of \eqref{NSubEq}.
\end{proof}
In the Galois case, the corollary means that one can find generators for $G$ (that satisfy a Demushkin relation)
with a prescribed action on carefully chosen elements of $\bar E$.

A proper open subgroup $U$ of $G$ is also a Demushkin group with
\begin{equation} \label{DemRankEq}
d(U) = \big(d(G)-2\big)[G : U] + 2
\end{equation}
so $d(U)$ is even as well.
We denote the invariants of $U$ by 
\begin{equation}
q^U, M_\kappa^U, \chi_\sigma^U, \chi^U, \tau^U,
\end{equation} 
while the invariants of $G$ will be denoted by 
\begin{equation}
q^G, M_\kappa^G, \chi_\sigma^G, \chi^G, \tau^G
\end{equation}
in order to avoid ambiguity.
From \cite[Proposition 18]{Ser}, we get that 
\begin{equation} \label{chiresEq}
\chi^U = \chi^G|_U
\end{equation}
which is obvious in the Galois case.

We denote by $\mathrm{res}^G_U$ the restriction map from $ M_\kappa^G$ to $M_\kappa^U$.
The subspace $N$ to which we will apply \corref{IsoCor} is given by the following.

\begin{cor} \label{IsotrCor}

The image of $\mathrm{res}^G_U$ is isotropic.

\end{cor}

The validity of this statement follows at once from the vanishing of the restriction map in the cohomology of a Demushkin group for dimension $2$ mentioned in \cite[Exercise 4.5.5 (a)]{Ser},
and from \cite[Proposition 1.5.3 (iii)]{NSW} telling us that cup product commutes with the restriction map.
We encourage the reader to translate \corref{IsotrCor} to the Galois case,
thus obtaining a vanishing result for the Hilbert symbol. 

In order to apply \corref{IsoCor} when $q \neq 0$, we need the following claim.

\begin{prop} \label{HeredProp}

Let $K$ be a subgroup of $G$ such that $d(G) > d(K) + 2$, and suppose that $q^G \neq 0$. 
Then there exists a proper open subgroup $U$ of $G$ containing $K$,
such that $\rho(\chi_\sigma^U) \in \mathrm{Im}\big(\mathrm{res}^G_U)$.

\end{prop}

\begin{proof}
Since $\mathrm{Im}(\chi^G) \leq_c \mathbb{U}_p^1$ there exist $g_1,g_2 \in G$ with
\begin{equation} \label{g1g2Eq}
\big\langle \chi^G(g_1), \chi^G(g_2) \big\rangle = \mathrm{Im}(\chi^G).
\end{equation}
In light of our assumption on the number of generators,
the subgroup 
\begin{equation} \label{subQDefEq}
Q \defeq \big\langle K \cup \{g_1, g_2\} \big\rangle
\end{equation}
is proper,
so there exists a proper open subgroup $U$ containing $Q$.
By \eqref{chiresEq},
\begin{equation}
\mathrm{Im}(\chi^U) \stackrel{\ref{chiresEq}}{=} \mathrm{Im}(\chi^G|_U) \supseteq \mathrm{Im}(\chi^G|_Q) 
\stackrel{\ref{g1g2Eq}}{=} \mathrm{Im}(\chi^G)
\end{equation}
and thus $\mathrm{Im}(\chi^U) = \mathrm{Im}(\chi^G)$.
Therefore, $q^G = q^U$ and $\tau^G = \tau^U$ so we have
\begin{equation}
\rho(\chi_\sigma^U) \stackrel{\ref{DescrEq}}{=} 
\tau^U \circ \chi^U = 
\tau^G \circ \chi^U \stackrel{\ref{chiresEq}}{=}
\mathrm{res}^G_U(\tau^G \circ \chi^G).
\end{equation}
\end{proof}

In the Galois case, the subgroup $U$ corresponds to a finite extension $E_0$ of $E$ contained in the fixed field of $K$,
such that $\zeta_{pq} \notin E_0$. 

\subsection{The proof of \thmref{FirstRes}}

We may assume that $d(G)$ is even,
since this is always the case for $q \neq 2$, 
and otherwise follows from \eqref{DemRankEq} after replacing $G$ by a subgroup of index $2$ containing $K$.
Since $K$ is a finitely generated pro-$p$ group, its Frattini quotient $K/\Phi(K)$ is finite,
so it follows from \cite[Lemma 1.2.5 (c)]{FJ} that (after passing to an open subgroup of $G$ containing $K$)
we can extend the quotient map $\varphi \colon K \to K/\Phi(K)$ to a homomorphism $\theta \colon G \to K/\Phi(K)$.
Upon passing to an open subgroup once again,
we may assume in view of \eqref{DemRankEq} that $d(G) > d(K) + 2$,
so by \propref{HeredProp} there exists a proper open subgroup $U$ of $G$ containing $K$ with 
\begin{equation} \label{PastEq}
\rho(\chi_\sigma^U) \in \mathrm{Im}\big(\mathrm{res}^G_U) \quad (\mathrm{if} \ q \neq 0).
\end{equation}

Set
\begin{equation}
\begin{split}
&\lambda \defeq \theta|_U, \quad \lambda^* \defeq \mathrm{inf}^{U/\Ker (\lambda)}_U, \quad 
\lambda^* \colon H^1(K/\Phi(K),\kappa) \to M_\kappa^U, \\
&N \defeq \mathrm{Im}(\lambda^*) \ (\mathrm{if} \ q =0), \
N \defeq \mathrm{Im}(\lambda^*) + \mathrm{Span}_\kappa \{\rho(\chi_\sigma^U)\} \ (\mathrm{otherwise}).
\end{split}
\end{equation}
Using the fact that $\lambda = \theta|_U$ and \eqref{PastEq} we see that
$N \subseteq \mathrm{Im}\big(\mathrm{res}^G_U)$, 
so $N$ is isotropic by \corref{IsotrCor}.
Hence, \corref{IsoCor} gives us a presentation 
\begin{equation}
\langle x_1, y_1, \dots, x_{t}, y_t \ | \ x_1^\gamma[x_1, y_1]x_2^\delta[x_2,y_2][x_3,y_3] \cdots [x_{t}, y_{t}] = 1 \rangle
\end{equation}
for $U$, such that for every $\psi \in \mathrm{Im}(\lambda^*)$ and $1 \leq i \leq t$ we have $\psi(x_i) = 0$.
This means that for every $\xi \in H^1(K/\Phi(K), \kappa)$ we have 
\begin{equation}
\forall \ 1 \leq i \leq t \quad \xi\big(\lambda(x_i)\big) = \big(\lambda^*(\xi)\big)(x_i) = 0.
\end{equation}
A vector in $K/\Phi(K)$ killed by all functionals (elements of $H^1\big(K/\Phi(K), \kappa \big)$) is zero, 
so $\lambda(x_i) = 0$ for all $1 \leq i \leq t$. 

By \lemref{EmbProbLem}, in order to show that $K$ is a retract of $U$, 
we need to find a weak solution $\mu$ to the following embedding problem
\begin{equation*}
\begin{tikzpicture}[scale=1.5]
\node (B) at (1.4,1) {$U$};
\node (C) at (0,0) {$K$};
\node (D) at (1.4,0) {$K/\Phi(K).$};
\path[->,font=\scriptsize,>=angle 90]
(B) edge node[right]{$\lambda$} (D)
(C) edge node[above]{$\varphi$} (D);
\end{tikzpicture}
\end{equation*} 
For all $1 \leq i \leq t$ define $\mu(x_i)$ to be $1 \in K$ and $\mu(y_i)$ to be an arbitrary element of $\varphi^{-1}\big(\lambda(y_i)\big)$. 
Then
\begin{equation}
\mu(x_1)^\gamma[\mu(x_1), \mu(y_1)]\mu(x_2)^\delta[\mu(x_2),\mu(y_2)] \cdots [\mu(x_t),\mu(y_t)] = 1
\end{equation}
so we can extend $\mu$ (uniquely) to a homomorphism from $U$ to $K$.
We have
\begin{equation}
\forall \ 1 \leq i \leq t \quad \varphi\big(\mu(x_i)\big) = \lambda(x_i), \ \varphi\big(\mu(y_i)\big) = \lambda(y_i)
\end{equation}
and we conclude that $\varphi \circ \mu = \lambda$ so $\mu$ is a weak solution.

\subsection{Consequences of \thmref{FirstRes}}

Note that \thmref{FirstRes} also gives a partial positive answer to \cite[Question 9.2]{KZ} on virtual retraction in groups from $\mathcal{L}$.
In order to prove \corref{ImmCor} we need the following claim.

\begin{prop} \label{intersectProp}

Let $G$ be a Demushkin group with $d(G) > 2$.
Then the intersection of any finite family of nontrivial normal subgroup of $G$ is nontrivial.

\end{prop}

\begin{proof}

By induction, it suffices to take some $\{1\} \neq A,B \lhd_c G$ and to show that they intersect nontrivially.
Suppose on the contrary that $A \cap B = \{1\}$ and take some nontrivial $a \in A, b \in B$.
It follows from \eqref{DemRankEq} that $\langle a \rangle , \langle b \rangle$ are subgroups of infinite index,
so by \cite[Exercise 4.5.5 (b)]{Ser} these subgroups are free.
Since $[a,b] \in A \cap B = \{1\}$,
it follows that $\langle a,b \rangle \cong \langle a \rangle \times \langle b \rangle \cong \mathbb{Z}_p^2$.
We infer once again from \eqref{DemRankEq} and from \cite[Exercise 4.5.5 (b)]{Ser} that $\langle a,b \rangle$ is a free pro-$p$ group,
contradicting the fact that it is isomorphic to $\mathbb{Z}_p^2$.
\end{proof}

Let us now prove \corref{ImmCor}.

\begin{proof}

Let $K$ be a finitely generated subgroup of infinite index in a Demushkin group $G$ with $d(G) > 2$.
By \thmref{FirstRes}, there exists an open subgroup $U$ of $G$ containing $K$,
and a normal subgroup $M$ of $U$ such that $K \cap M = \{1\}$ and $KM = U$.
Since the index of $K$ in $G$ is infinite, it follows immediately that $M$ is infinite.
As $U$ is open in $G$, we can find an open subgroup $V$ of $U$ that is normal in $G$,
and conclude that $R \defeq M \cap V$ is a nontrivial normal subgroup of $V$.
It follows from \propref{intersectProp} that 
\begin{equation}
N \defeq \bigcap_{g \in G/V} g^{-1}Rg \neq \{1\}.
\end{equation}
Hence, $N$ is a nontrivial normal subgroup of $G$ intersecting $K$ trivially.
\end{proof}

We also prove \corref{adicSylowCor}.

\begin{proof}

Let $G$ be the $p$-Sylow subgroup of the absolute Galois group of $\mathbb{Q}_p$,
and let $H$ be a finitely generated subgroup.
It follows from \cite{Lab0} that 
\begin{equation}
G = \varprojlim_{n \in  \mathbb{N}} G_n
\end{equation}
where each $G_n$ is a Demushkin group, and $d(G_n) \to \infty$ as $n \to \infty$.
In view of \eqref{DemRankEq}, there exists an $N \in \mathbb{N}$ such that for every $n \geq N$, 
the image of $H$ in $G_n$ is of infinite index, and thus free.
Taking $m \geq N$ large enough, we may assume that the number of generators of the image of $H$ in $G_m$ coincides with $d(H)$ and that $d(G_m) > 2$.
By the Hopfian property \cite[Proposition 2.5.2]{RZ}, the projection of $H$ to $G_m$ is injective
and we invoke \thmref{FirstRes}.
\end{proof}

\section{Howson's theorem for Demushkin groups} \label{HowsSec}

Let us now establish Howson's theorem for Demushkin groups, that is, \thmref{SecRes}.
Our approach is based on the proof of \cite[Lemma 3.1]{Sh}.

\begin{thm} \label{EHowsThm}

Let $p$ be a prime, let $G$ be a pro-$p$ Demushkin group, 
and let $A,B \leq_c G$ be nontrivial finitely generated subgroups. Then
\begin{equation} \label{HowsGoalEq}
d(A \cap B) \leq \Big(p^2\big(d(A) + d(B) - 2\big)^2\Big)\big(d(A)-1\big)\big(d(B)-1\big) +  1.
\end{equation}

\end{thm}

\begin{proof}

Note that if $d(G) \leq 2$ (or equivalently, if $G$ is solvable) then by \eqref{DemRankEq}, 
every open subgroup $U$ of $G$ satisfies $d(U) \leq 2$.
Hence, by \cite[Proposition 3.11]{DDMS}, for every $H \leq_c G$ we have $d(H) \leq 2$ so \eqref{HowsGoalEq} holds.
In light of that, we assume in what follows that $d(G) \geq 3$.

If either $A$ or $B$ is open in $G$ (say, $[G : B] < \infty$),
then by \cite[Corollary 3.6.3]{RZ} (Schreier's bound) we have
\begin{equation}
\begin{split}
d(A \cap B) &\leq \big(d(A)-1\big)[A : A \cap B] + 1 \\
&\leq \big(d(A)-1\big)[G : B] + 1 
\stackrel{\ref{DemRankEq}}{\leq} \big(d(A) - 1\big)\big(d(B) - 2\big) + 1.
\end{split}
\end{equation}

We may thus assume that $A$ and $B$ are of infinite index in $G$, 
so there exist strictly descending chains $\{U_k\}_{k=0}^\infty$ and $\{V_k\}_{k=0}^\infty$ of open subgroups of $G$ intersecting at $A$ and $B$ respectively.
After refining the chains if necessary, 
we may assume that $U_0 = V_0 = G$, 
that $U_{k+1}$ is a maximal subgroup of $U_k$,
and that $V_{k+1}$ is a maximal subgroup of $V_k$ for every  $k \geq 0$.
Hence,
\begin{equation}
[U_k : U_{k+1}] = [V_k : V_{k+1}] = p
\end{equation}
so we have
\begin{equation} \label{ExpIndEq}
[G : U_k] = [G : V_k] = p^k.
\end{equation}
If $A$ is procyclic then \eqref{HowsGoalEq} holds trivially, so we may assume that $d(A) \geq 2$ and define
\begin{equation} \label{HowsDefEq}
\begin{split}
&n \defeq \Big\lfloor\log_p\ \big(d(A) + d(B) - 2\big) \Big\rfloor + 1 \\ 
&A_0 \defeq A \cap V_n, \quad B_0 \defeq B \cap U_n, \quad C \defeq \langle A_0 \cup B_0 \rangle \leq U_n \cap V_n,
\end{split}
\end{equation}
so that
\begin{equation} \label{BitLongEq}
\begin{split}
d(C) &\leq d(A_0) + d(B_0) \leq (d(A)-1)[A : A_0] + (d(B)-1)[B : B_0] + 2 \\
&\stackrel{\ref{HowsDefEq}}{\leq} (d(A)-1)[U_n : U_n \cap V_n] + (d(B)-1)[V_n : U_n \cap V_n] + 2 \\
&= [G : U_n \cap V_n]\Bigg(\frac{d(A)-1}{[G : U_n]} + \frac{d(B)-1}{[G : V_n]}\Bigg) + 2\\
&\stackrel{\ref{ExpIndEq}}{=} [G : U_n \cap V_n]\Bigg(\frac{d(A) + d(B) - 2}{p^n}\Bigg) + 2\\
&\stackrel{\ref{HowsDefEq}}{<} [G : U_n \cap V_n] + 2 \leq (d(G)-2)[G : U_n \cap V_n] + 2
\end{split}
\end{equation}
where the last inequality follows from our assumption that $d(G) \geq 3$.

If $C$ was an open subgroup of $G$, we would have
\begin{equation} 
d(C) \stackrel{\ref{DemRankEq}}{=} \big(d(G)-2\big)[G : C] + 2 \stackrel{\ref{HowsDefEq}}{\geq} \big(d(G)-2\big)[G : U_n \cap V_n] + 2 
\end{equation}
contrary to \eqref{BitLongEq}.
We conclude that $[G : C]$ is infinite, so $C$ is a free pro-$p$ group.
Applying \cite[Theorem 1.1]{J} (the Hanna Neumann Conjecture for free pro-$p$ groups)
to the nontrivial subgroups $A_0,B_0 \leq_c C$ we get that
\begin{equation}
\begin{split}
d(A \cap B) &\stackrel{\ref{HowsDefEq}}{=} d(A_0 \cap B_0) \leq \big(d(A_0)-1\big)\big(d(B_0)-1\big) + 1 \\
&\leq \big(d(A)-1\big)[A : A_0]\big(d(B)-1\big)[B : B_0] + 1 \\
&\stackrel{\ref{HowsDefEq}}{\leq} \big(d(A)-1\big)[G : V_n]\big(d(B)-1\big)[G : U_n] + 1 \\
&\stackrel{\ref{ExpIndEq}}{=} p^{2n}\big(d(A)-1\big)\big(d(B)-1\big) + 1 \\
&\stackrel{\ref{HowsDefEq}}{\leq} p^2\big(d(A) + d(B) - 2\big)^2\big(d(A)-1\big)\big(d(B)-1\big) +  1.
\end{split}
\end{equation}
\end{proof}

\thmref{EHowsThm} gives a partial positive answer to \cite[Open Question 9.5.7]{RZ} and \cite[Question 9.4]{KZ} asking for Howson's theorem in groups from $\mathcal{L}$.

In the core of our proof, namely the case
\begin{equation}
d(G) \geq 3, \quad [G : A] = \infty, \quad [G:B] = \infty
\end{equation}
we have used a quantitative version of the argument appearing in the proof of \cite[Lemma 3.1]{Sh} in order to find finite index subgroups $A_0,B_0$ (of $A$ and $B$ respectively) that generate a subgroup of infinite index in $G$.
As in \cite[Lemma 3.1]{Sh}, we only need $G$ to have positive rank gradient (that is, linear growth of the number of generators of open subgroups $U$ of $G$ as a function of $[G : U]$) for the argument to work. 
In order to deduce Howson's theorem for $G$, 
we have used the fact that $G$ is IF (that is, every finitely generated subgroup of infinite index is free pro-$p$).
Thus, we have shown that IF-groups with positive rank gradient satisfy Howson's theorem.
Examples of such groups are given in \cite{SZ}, they are studied in \cite{Shus}, and \cite[Open Question 7.10.5]{RZ}  
asks about their properties.

Furthermore, our argument also applies with almost no changes to discrete finitely generated groups $G$ that are IF and have positive rank gradient. 
We only need to further assume that $G$ is LPF (that is, for every finitely generated subgroup $A$ of infinite index in $G$, there exists an infinite strictly descending chain of finite index subgroups of $G$ containing $A$ - see \cite{C, Sh}).
In particular, we get a new proof of Howson's theorem for surface groups (for the first proof, see \cite{Green}).
The bound on the number of generators of the intersection provided by such a proof depends on the indices of the subgroups in the descending chains above the given finitely generated subgroups $A$ and $B$.
For instance, in case that $A$ and $B$ are closed in the pro-$2$ topology on $G$,
inequality \eqref{HowsGoalEq} holds with $p=2$,
so we obtain an improvement on the best known bound from \cite{Som1} under the additional assumption that $d(A)$ and $d(B)$ are relatively small.
For general $A$ and $B$ one can derive bounds on $d(A \cap B)$ using quantitative versions of the LPF/LERF property of surface groups obtained in \cite[Theorem 7.1]{Pat}.

\section{Free products}

\subsection{Properties preserved by free products}

We need the following lemma for the proof of \thmref{PresVRThm} and \thmref{PresMHThm}.

\begin{lem} \label{ProdRepLem}

Let $A$ be a pro-$p$ group, let $B$ be a free pro-$p$ group, and set $G \defeq A\amalg B$.
Let $C$ be a subgroup of $G$ such that the restriction to $C$ of the projection from $G$ to $B$ is an isomorphism. 
Then $G = A \amalg C$.

\end{lem}

\begin{proof} 

A Frattini argument shows that $\langle A,C \rangle = G$ so by \cite[Theorem 4.2]{Mel},
it suffices to show that the map on homology  
\begin{equation}
\mathrm{Cor} \colon H_q(A, \mathbb{F}_p) \oplus H_q(C,\mathbb{F}_p) \rightarrow H_q(G,\mathbb{F}_p)
\end{equation}
induced by corestrictions is surjective for $q=2$ and injective for $q=1$.  
The latter follows from a Frattini argument, 
and for the former we use \cite[Theorem 4.1 (2)]{Mel} saying that
\begin{equation}
H_2(A, \mathbb{F}_p) \oplus H_2(B, \mathbb{F}_p) \stackrel{\mathrm{Cor}}{\cong} H_2(G, \mathbb{F}_p)
\end{equation}
and recall that $H_2(B, \mathbb{F}_p) \cong H_2(C, \mathbb{F}_p) \cong 0$ as $B$ and $C$ are free. \end{proof}

We are going to use throughout the pro-$p$ Kurosh subgroup theorem for finitely generated subgroups, and for open subgroups, as stated in \cite[Theorem 2.1]{Rib} and in \cite[Theorem 9.1.9]{RZ} respectively.
Let us reformulate and prove \thmref{PresVRThm} and \thmref{PresMHThm}. 

\begin{thm} \label{virtretrfreeprod} 

Let $n \in \mathbb{N}$,
and let $G_1, \dots, G_n$ be pro-$p$ groups with the virtual retraction property (respectively, the M. Hall property).
Then 
\begin{equation}
G \defeq \coprod_{i=1}^n G_i
\end{equation}
possesses the virtual retraction property (respectively, the M. Hall property).

\end{thm}

\begin{proof}  

Let $H$ be a finitely generated subgroup of $G$. 
By the Kurosh subgroup theorem (see \cite[Theorem 2.1]{Rib}),

\begin{equation} \label{HDecEq}
H=\coprod_{i=1}^n \coprod_{j=1}^{r_i} (G_i^{g_{i,j}} \cap H)\amalg F_H
\end{equation} 
where for each $1 \leq i \leq n$ the $g_{i,j}$ range over those representatives for the double cosets $G_i\backslash G/H$ for which $G_i^{g_{i,j}} \cap H \neq \{1\}$,
and $F_H$ is a finitely generated free pro-$p$ group.  
Since $H$ is finitely generated, 
the number of factors in \eqref{HDecEq} is finite (that is, $\forall i \ r_i < \infty$), 
and these factors are finitely generated.  
There exists an open subgroup $V$ of $G$ containing $H$ such that
for every open subgroup $U$ of $V$ containing $H$ we have
\begin{equation}
\forall \ 1 \leq i \leq n, 1 \leq j \neq k \leq r_i \quad G_ig_{i,j}U\neq G_ig_ {i, k}U
\end{equation}
and $G_i^{g_{i,j}} \cap H$ is a retract (respectively, a free factor) of $G_i^{g_{i,j}} \cap U$.
According to \cite[Theorem 9.1.9]{RZ}, we have the following Kurosh decomposition
\begin{equation} \label{KurDec}
U=\coprod_{i=1}^n \coprod_{j=1}^{r_i} (G_i^{g_{i,j}} \cap U)\amalg \coprod_{i=1}^n \coprod_{k=1}^{t_i} (G_i^{\beta_{i,k}} \cap U)\amalg F_U.
\end{equation} 

For any $K \leq_c G$ we define a closed subset and a (closed) subgroup by 
\begin{equation} \label{XDefEq}
X_K \defeq K \cap  \bigcup_{i=1}^n\bigcup_{g \in G} G_i^g, \quad \tilde K \defeq \langle X_K\rangle.
\end{equation}
We claim that if $K$ is either finitely generated or open in $G$,
and its Kurosh decomposition is given by
\begin{equation} \label{KuroshDecKEq}
K=\coprod_{i=1}^n \coprod_{j \in J_i} (G_i^{\alpha_{i,j}} \cap K)\amalg F_K
\end{equation}
(where $J_i$ is a set of representatives for $G_i\backslash G/K$) then for
\begin{equation} \label{YDefEq}
Y_K \defeq \bigcup_{k \in K} \Big( \bigcup_{i,j} (G_i^{\alpha_{i,j}} \cap K) \Big)^k
\end{equation}
we have $X_K = Y_K$.
Indeed, for every $1 \leq i \leq n$ and $g \in G$ there exists a $j \in J_i$ such that $g = t\alpha_{i,j}k$ for some 
$t \in G_i$ and $k \in K$.
Hence
\begin{equation}
G_i^g \cap K = G_i^{t\alpha_{i,j}k} \cap K = G_i^{\alpha_{i,j}k} \cap K = (G_i^{\alpha_{i,j}} \cap K)^k 
\stackrel{\ref{YDefEq}}{\subseteq} Y_K
\end{equation}
so $X_k \subseteq Y_k$ and our claim is established.
We thus have
\begin{equation} \label{FreeQuotEq}
K/\tilde K \stackrel{\ref{XDefEq}}{=} K/ \langle X_K\rangle = K/ \langle Y_K\rangle 
\stackrel{\ref{KuroshDecKEq}}{\cong} F_K.
\end{equation}

By \eqref{XDefEq},
\begin{equation} \label{InterXEq}
X_H = \bigcap_{H \leq U \leq_o V} X_U
\end{equation}
and we would like to show that 
\begin{equation} \label{limitEq}
\tilde H=\bigcap_{H\leq U\leq_o V} \tilde U.
\end{equation}
For that pick a homomorphism $\varphi$ from $G$ onto a finite group.
Finiteness implies that there exists an $L \leq_o V$ containing $H$ such that
\begin{equation}
\begin{split}
\varphi(\tilde H) &\stackrel{\ref{XDefEq}}{=} \varphi\Big(\langle X_H \rangle\Big) =  
\Big\langle \varphi(X_H) \Big\rangle 
\stackrel{\ref{InterXEq}}{=} \Big\langle \varphi\Big(\bigcap_{H \leq U \leq_o V} X_U\Big) \Big\rangle \\
&= \Big\langle \bigcap_{H \leq U \leq_o V} \varphi(X_U) \Big\rangle = \Big\langle \varphi(X_{L}) \Big\rangle
=  \varphi \Big(\langle X_{L} \rangle \Big) \\
&\stackrel{\ref{XDefEq}}{=} \varphi(\tilde L) \supseteq \varphi \Big( \bigcap_{H\leq U\leq_o V} \tilde U \Big)
\end{split}
\end{equation}  
so \eqref{limitEq} holds. 
It follows that
\begin{equation}
F_H \stackrel{\ref{FreeQuotEq}}{\cong} H/\tilde H \stackrel{\ref{limitEq}}{\cong} \varprojlim_{H\leq U\leq_o V}  U/\tilde U\stackrel{\ref{FreeQuotEq}}{\cong} \varprojlim_{H\leq U\leq_o V} F_U
\end{equation}
and we denote by $\eta_U \colon F_H \to F_U$ the projections from the inverse limit.

Since $F_H$ is finitely generated,
there is an $M \leq_o V$ such that 
\begin{equation}
d\big(\eta_M(F_H)\big) = d(F_H),
\end{equation}
and $\eta_M(F_H) \leq F_M$ is a free pro-$p$ group, so $\eta_M$ is an injection in view of the Hopfian property 
(see \cite[Proposition 2.5.2]{RZ}).
By \cite[Theorem 9.1.19]{RZ}, there exists an $R \leq_o F_M$ such that $\eta_M(F_H)$ is a free factor (and thus also a retract) of $R$.
Taking $N$ to be the preimage of $R$ under the quotient map $M \to M/ \tilde M = F_M$,
we note that the surjection $N \to R$ factors through $F_N$.
We conclude that $\eta_N(F_H)$ is a retract (and thus also a free factor, by a Frattini argument) of $F_N$,
since the homomorphism $F_N \to F_M$ maps the former injectively onto a retract of the image of the latter.
Hence
\begin{equation}
F_N = \eta_N(F_H) \amalg Q
\end{equation}
for some (free pro-$p$ subgroup) $Q \leq_c F_N$.

Set $F_0 \defeq \langle F_H \cup Q \rangle$
and note that since the images of $F_H$ and $Q$ under the projection $N \to N/\tilde N \cong F_N$ generate a free product, we must have 
\begin{equation}
F_0 = F_H \amalg Q.
\end{equation}
By \lemref{ProdRepLem}, we can replace $F_N$ by $F_0$ in the Kurosh decomposition \eqref{KurDec} of $U = N$,
so that every factor in the decomposition \eqref{HDecEq} of $H$ is a retract (respectively, a free factor) of a factor in the decomposition \eqref{KurDec} of $N$.
It follows that $H$ is a retract (respectively, a free factor) of $N$.
\end{proof}

The following is an immediate consequence of the proof of \thmref{virtretrfreeprod}.

\begin{cor} \label{SubProdCor}

Let $G_1, \dots, G_n$ be pro-$p$ groups,
set 
\begin{equation}
G \defeq \coprod_{i=1}^n G_i
\end{equation}
and let $H$ be a finitely generated subgroup of $G$.
Then there exists an open subgroup $U$ of $G$ containing $H$,
and subgroups $U_1, \dots, U_m$ of $U$ such that
\begin{equation}
U = \coprod_{i=1}^m U_i, \quad H = \coprod_{i=1}^m H_i
\end{equation}
where for each $1 \leq i \leq m$ we have $H_i \leq_c U_i$.

\end{cor} 

We shall need some auxiliary results for the proof of \thmref{BaumThm}.
The first is just a variant of Shapiro's lemma.

\begin{prop} \label{ShapiroProp}

Let $G$ be a pro-$p$ group acting continuously on a profinite space $X$,
and let $S \subseteq X$ be a finite subset for which the orbits $\{O_s\}_{s \in S}$ are pairwise distinct.
Then for the disjoint union $T$ of these orbits we have
\begin{equation}
H_1\big(G,\llbracket \mathbb{F}_pT \rrbracket\big) \cong \bigoplus_{s \in S} H_1(G_s, \mathbb{F}_p).
\end{equation}

\end{prop}

\begin{proof}

Applying Shapiro's lemma (see \cite[Theorem 6.10.8 (d)]{RZ}) we get
\begin{equation}
\begin{split}
H_1\big(G,\llbracket \mathbb{F}_pT \rrbracket\big) &= 
H_1\big(G,\llbracket \mathbb{F}_p \bigcup_{s \in S} O_s \rrbracket\big) \cong 
H_1\big(G, \bigoplus_{s \in S} \ \llbracket \mathbb{F}_p O_s \rrbracket\big) \\
&\cong \bigoplus_{s \in S} H_1\big(G,\llbracket \mathbb{F}_p O_s \rrbracket\big) 
\cong \bigoplus_{s \in S} H_1\big(G,\llbracket \mathbb{F}_p (G/G_s) \rrbracket\big) \\
&\cong \bigoplus_{s \in S} H_1(G, \mathrm{Ind}_{G_s}^{G} \mathbb{F}_p) \cong
\bigoplus_{s \in S} H_1(G_s, \mathbb{F}_p).
\end{split}
\end{equation}
\end{proof}

\begin{prop} \label{InjectProp}

Let $G$ be a pro-$p$ group acting continuously on a profinite space $X$,
and let $R \subseteq X$ be a closed $G$-invariant subset.
Then the natural inclusion $R \hookrightarrow X$ induces an embedding
\begin{equation}
\eta_R \colon H_1\big(G,\llbracket \mathbb{F}_pR\rrbracket\big) \hookrightarrow 
H_1\big(G,\llbracket \mathbb{F}_pX\rrbracket\big).
\end{equation} 

\end{prop}

\begin{proof}

Let $U$ be a clopen $G$-invariant subset of $X$ containing $R$. 
Then
\begin{equation}
\llbracket \mathbb{F}_pX \rrbracket = \llbracket\mathbb{F}_pU \rrbracket\oplus 
\llbracket\mathbb{F}_p(X\setminus U)\rrbracket
\end{equation} 
hence 
\begin{equation}
H_1\big(G,\llbracket \mathbb{F}_pX\rrbracket\big) = 
H_1\big(G,\llbracket\mathbb{F}_pU\rrbracket\big) \oplus 
H_1\big(G,\llbracket\mathbb{F}_p(X\setminus U)\rrbracket\big)
\end{equation}
and in particular the inclusion $U\hookrightarrow X$ induces an embedding 
\begin{equation}
\eta_U \colon H_1\big(G,\llbracket\mathbb{F}_pU\rrbracket\big) \hookrightarrow 
H_1\big(G,\llbracket \mathbb{F}_pX\rrbracket\big).
\end{equation}
By \cite[Lemma 5.6.4 (a)]{RZ}, $R$ is an intersection of clopen $G$-invariant subsets of $X$, 
so by \cite[Lemma 3.2]{Mel} (which says that homology commutes with inverse limits in the second variable) we have
\begin{equation}
\eta_R = \varprojlim_U \eta_U
\end{equation} 
where the limit is taken over the clopen $G$-invariant subsets of $X$ containing $R$.
It follows that $\eta_R$ is an inverse limit of injections, so it is an injection.
\end{proof}

\begin{lem} \label{finiteness}

Let $G$ be a pro-$p$ group acting on a profinite space $X$ with finitely generated point stabilizers. 
Then $H_1(G, \llbracket \mathbb{F}_pX\rrbracket )$ is finite if and only if there are only finitely many $G$-orbits (in $X$) with a non-free $G$-action.

\end{lem}

\begin{proof} 

Let $T \subseteq X$ be the set of points with non-trivial $G$-stabilizers, that is
\begin{equation}
T \defeq \big\{x \in X \ | \ G_x \neq \{1\} \big\}.
\end{equation} 
If $G \backslash T$ is finite then $T$ is closed in $X$ so by \cite[Exercise 5.2.4 (b)]{RZ} we have
\begin{equation}
\llbracket \mathbb{F}_pX \rrbracket / \llbracket \mathbb{F}_pT \rrbracket \cong
\llbracket \mathbb{F}_p(X/T,*) \rrbracket 
\end{equation}
and we denote this profinite $\llbracket \mathbb{F}_pG \rrbracket$-module by $M$ (here $(X/T,*)$ is a pointed profinite space with $T=*$ as the distinguished point).
Clearly, $G$ acts freely on the pointed profinite space $(X/T,*)$, so by \cite[Proposition 2.3]{ZM} $M$ is a projective $\llbracket \mathbb{F}_pG \rrbracket$-module.
By \cite[Corollary 7.5.4]{Wilson}, $M$ is free,
so from \cite[Proposition 6.3.4 (b)]{RZ} we get that $H_n(G, M)=0$ for $n \geq 1$. 

Proposition 6.3.4 (c) from \cite{RZ} associates to the short exact sequence
\begin{equation}
0 \rightarrow \llbracket \mathbb{F}_pT \rrbracket \rightarrow 
\llbracket \mathbb{F}_pX \rrbracket \rightarrow M \rightarrow 0
\end{equation} 
a long exact sequence of homology
\begin{equation}
0 = H_{2}(G,M) \rightarrow 
H_1(G,\llbracket \mathbb{F}_pT \rrbracket) \rightarrow 
H_1(G,\llbracket \mathbb{F}_pX \rrbracket) \rightarrow 
H_1(G,M)=0
\end{equation} 
from which we conclude that
\begin{equation} \label{TXCompEq}
H_1(G,\llbracket \mathbb{F}_pT \rrbracket)\cong H_1(G,\llbracket \mathbb{F}_pX \rrbracket).
\end{equation}  
Taking a section $S$ of $G \backslash T$ in $T$,
and applying \propref{ShapiroProp}
we find that
\begin{equation}
H_1(G,\llbracket \mathbb{F}_pX \rrbracket) \stackrel{\ref{TXCompEq}}{\cong} 
H_1(G,\llbracket \mathbb{F}_pT \rrbracket) \cong 
\bigoplus_{s \in S} H_1(G_s, \mathbb{F}_p)
\end{equation} 
and conclude that $H_1(G,\llbracket \mathbb{F}_pX \rrbracket)$ is finite since by assumption, $S$ is finite and $G_s$ is finitely generated for each $s \in S$.

Conversely, assume that there is an $n \in \mathbb{N}$ with
$|H_1(G,\llbracket \mathbb{F}_pX \rrbracket)| < n$ 
and toward a contradiction, suppose that there are $n$ distinct $G$-orbits 
\begin{equation}
O_{x_1}, \dots, O_{x_n}
\end{equation}
with nontrivial point stabilizers. 
We denote by $R$ the union of these orbits, and invoke \propref{ShapiroProp} and \propref{InjectProp}
to obtain a contradiction:
\begin{equation}
n > |H_1(G,\llbracket \mathbb{F}_pX \rrbracket)| \stackrel{\ref{InjectProp}}{\geq} 
|H_1(G,\llbracket \mathbb{F}_pR \rrbracket)| \stackrel{\ref{ShapiroProp}}{=} 
|\bigoplus_{i = 1}^n H_1(G_{x_i}, \mathbb{F}_p)| \geq n.
\end{equation} 
\end{proof}

In light of \corref{SubProdCor}, it suffices to prove the following in order to establish \thmref{BaumThm}.

\begin{thm} \label{H fitted} 

Let $G_1, \dots, G_n$ be pro-$p$ groups satisfying Howson's property,
and let $H_i \leq_c G_i$ be finitely generated subgroups for $1 \leq i \leq n$.
Set
\begin{equation} \label{FreeProdGHEq}
G \defeq \coprod_{i=1}^n G_i, \quad H \defeq \coprod_{i=1}^n H_i,
\end{equation}
and let $K$ be a finitely generated subgroup of $G$.
Then $d(K \cap H) < \infty$.

\end{thm}

\begin{proof} 

Consider the exact sequences
\begin{equation} \label{GXctSeqEq}
0 \rightarrow  \llbracket\mathbb{F}_pG\rrbracket^{n-1} \rightarrow 
\bigoplus_{i=1}^{n} \llbracket \mathbb{F}_p(G/G_i)\rrbracket \rightarrow 
\mathbb{F}_p\rightarrow 0,
\end{equation}
\begin{equation} \label{HXctSeqEq}
0\rightarrow \llbracket\mathbb{F}_pH\rrbracket^{n-1} \rightarrow
\bigoplus_{i=1}^{n}\llbracket\mathbb{F}_p(H/H_i)\rrbracket\rightarrow \mathbb{F}_p\rightarrow 0
\end{equation} 
functorially associated with the free pro-$p$ products from \eqref{FreeProdGHEq} by \cite[Theorem 4.1, proof of Proposition 2.7]{RZ2}. 
Applying the functor $\mathbb{F}_p\hat\otimes_{\llbracket\mathbb{F}_pK\rrbracket} -$ to \eqref{GXctSeqEq} 
and the functor $\mathbb{F}_p\hat\otimes_{\llbracket\mathbb{F}_p(K\cap H)\rrbracket} -$ to \eqref{HXctSeqEq}
we get the associated long homological exact sequences 
(these are instances of the Mayer-Vietoris sequences, see \cite[Proposition 6.3.4 (c)]{RZ}) 
\begin{equation} \label{FLongXSeqEq}
\begin{split}
0 \rightarrow &\bigoplus_{i=1}^{n} H_1\big(K, \llbracket \mathbb{F}_p (G/G_i) \rrbracket\big) 
\rightarrow H_1(K, \mathbb{F}_p) \rightarrow 
\llbracket \mathbb{F}_p (K\backslash G) \rrbracket^{n-1} \overset{\delta}{\rightarrow} \\
&\bigoplus_{i=1}^{n} \llbracket \mathbb{F}_p (K\backslash G/G_i) \rrbracket \rightarrow \mathbb{F}_p \rightarrow 0,
\end{split}
\end{equation}
\begin{equation} \label{SLongXSeqEq}
\begin{split}
0 \rightarrow &\bigoplus_{i=1}^{n} H_1\big(K \cap H, \llbracket \mathbb{F}_p (H/H_i) \rrbracket\big) 
\rightarrow H_1(K \cap H,\mathbb{F}_p) \rightarrow \\
&\llbracket \mathbb{F}_p (K\cap H)\backslash H \rrbracket^{n-1} \overset{\sigma}{\rightarrow}  
\bigoplus_{i=1}^{n} \llbracket\mathbb{F}_p (K \cap H) \backslash H/H_i \rrbracket \rightarrow \mathbb{F}_p \rightarrow 0.
\end{split}
\end{equation}  
The fact that the sequences start with a $0$ follows from the freeness of the modules 
$\llbracket\mathbb{F}_pG \rrbracket, \llbracket\mathbb{F}_pH \rrbracket$ (see \cite[Corollary 5.7.2 (a), Proposition 6.3.4 (b)]{RZ}).
Applying the functor $H_0(K \cap H, -)$ to \eqref{GXctSeqEq} and \eqref{HXctSeqEq},
and using the morphism $H_0(K \cap H, -) \to H_0(K,-)$ of functors (from $G$-modules to $\mathbb{F}_p$-spaces), 
we obtain the following commutative diagram
\begin{equation}
\xymatrix{\llbracket\mathbb{F}_p(K\backslash G)\rrbracket^{n-1} \ar[r]^>>>>>>>>>>{\delta} &\bigoplus_{i=1}^{n}\llbracket\mathbb{F}_p(K\backslash G/G_i)\rrbracket\\
            \llbracket\mathbb{F}_p(K\cap H)\backslash H\rrbracket^{n-1} \ar[r]^<<<<{\sigma}\ar[u]^{\alpha} &\bigoplus_{i=1}^{n}\llbracket\mathbb{F}_p(K\cap H)\backslash H/H_i\rrbracket\ar[u]^{\beta}}
\end{equation}
and we want to show that $\mathrm{Ker}(\sigma)$ and $\mathrm{Ker}(\beta)$ are finite. 

Exactness of \eqref{FLongXSeqEq} implies that $\mathrm{Ker}(\delta)$ is an image of $H_1(K,\mathbb{F}_p)$ so it is finite as $K$ is finitely generated. 
Since $\alpha$ is an injection, it follows at once that $\mathrm{Ker}(\delta\alpha)$ is finite as well. 
Commutativity tells us that $\mathrm{Ker}(\delta\alpha) = \mathrm{Ker}(\beta\sigma)$
so the latter is finite, and thus its subspace $\mathrm{Ker}(\sigma)$ is also finite.
Furthermore,
\begin{equation}
\mathrm{Ker}(\beta) \cap \mathrm{Im}(\sigma) = \sigma\big(\mathrm{Ker}(\beta\sigma)\big)
\end{equation}
so
\begin{equation}
|\mathrm{Ker}(\beta) \cap \mathrm{Im}(\sigma)| = | \sigma\big(\mathrm{Ker}(\beta\sigma)\big)| \leq
| \mathrm{Ker}(\beta\sigma)| < \infty.
\end{equation}
Exactness of \eqref{SLongXSeqEq} implies that $\mathrm{Im}(\sigma)$ has codimension $1$,
so $\mathrm{Ker}(\beta) \cap \mathrm{Im}(\sigma)$ is of codimension at most $1$ in $\mathrm{Ker}(\beta)$ 
and we conclude that the latter is finite. 

For each $1 \leq i \leq n$ consider the action of $K$ on $G/G_i$,
and note that every point stabilizer is the intersection of $K$ with a conjugate of $G_i$,
which by the Kurosh subgroup theorem, is a free factor of $K$ and thus finitely generated.
Moreover, the finiteness of $H_1(K, \llbracket\mathbb{F}_p(G/G_i)\rrbracket)$            
follows from \eqref{FLongXSeqEq}, so by \lemref{finiteness}, $K$ acts freely on the complement of a finite set of orbits, and thus $K \cap H$ acts on this complement freely as well.

The fact that the fibers of $\beta$ are finite implies that only finitely many $K \cap H$-orbits in $H/H_i$ are mapped to each $K$-orbit in $G/G_i$ under the natural inclusion $H/H_i \hookrightarrow G/G_i$,
so we deduce that there are only finitely many $H\cap K$-orbits in $H/H_i$ with a non-free action.
For every point stabilizer $L$ for the action of $K \cap H$ on $H/H_i$ there exists an $h \in H$ such that
\begin{equation}
L = K \cap H \cap H_i^h = K \cap H_i^{h} = (K \cap G_i^{h}) \cap H_i^{h}
\end{equation}
which is finitely generated since $G_i$ satisfies Howson's property.          
We can now invoke \lemref{finiteness} to conclude that $H_1(K \cap H, \llbracket \mathbb{F}_p(H/H_i)\rrbracket)$ is finite. 
Combining this with the finiteness of $\mathrm{Ker}(\sigma)$ and the exactness of \eqref{SLongXSeqEq},
we see that $H_1(K \cap H, \mathbb{F}_p)$ is finite, as required.
\end{proof}

\subsection{Finitely generated M. Hall groups}

%
%
%
%
%
%
%
%
%
%
%
%

We begin by recalling from \cite[Lemma 4.3]{Rib} that the M. Hall property is hereditary.

\begin{prop} \label{HallHeredProp}

Let $G$ be an M. Hall pro-$p$ group, and let $H$ be a finitely generated subgroup. 
Then $H$ is M. Hall as well.

\end{prop}

The proof of the following lemma shows that a pro-$p$ M. Hall group satisfies the ascending chain condition on finite subgroups.

\begin{lem} \label{maximal finite} 
Any pro-$p$ M. Hall group $G$ has a maximal finite subgroup.
\end{lem}

\begin{proof}

Let $K$ be a nontrivial finite subgroup of $G$.
Since $G$ is M. Hall, there exists an open subgroup $U$ of $G$ such that $U = K \amalg L$ for some $L \leq_c U$.

We claim that $K$ is a maximal finite subgroup of $U$.
Indeed, if $M$ is a finite subgroup of $U$ that contains $K$, 
then by the Kurosh subgroup theorem $K$ is a free factor of $M$.
Hence if toward a contradiction $K \lneq M$, the latter is a nontrivial free pro-$p$ product, so it is infinite by \cite[Theorem 9.1.6]{RZ}. 
 
Therefore, if $R$ is a finite subgroup of $G$ containing $K$ then $R \cap U = K$. 
We conclude that $[R : K] \leq [G : U]$ and thus $|R| \leq |K|[G : U]$.
\end{proof}

On our way to showing that finitely generated freely indecomposable infinite pro-$p$ M. Hall groups are torsion-free, 
we need the following claim. 

\begin{prop} \label{index p} 

Let $G$ be a finitely generated freely indecomposable infinite pro-$p$ M.Hall group, 
and let $Q$ be a nontrivial maximal finite subgroup. 
Then there exists a finitely generated freely indecomposable infinite subgroup $K$ of $G$ having a subgroup $R$ of index $p$ such that $Q$ is a free factor of $R$.

\end{prop}

\begin{proof}

Take an open subgroup $V$ of $G$ containing $Q$ as a free factor with $[G : V]$ as small as possible. 
Since $G$ is freely indecomposable, $V$ is a proper subgroup, 
so there exists an open subgroup $U$ of $G$ containing $V$ as a subgroup of index $p$.
We can write $U$ as a free pro-$p$ product 
\begin{equation}
U = U_1 \amalg U_2 \cdots \amalg U_n
\end{equation}
of freely indecomposable factors.
By \cite[Theorem 4.2 (a)]{RZ2}, $Q$ is conjugate to a subgroup of one of these factors,  
so we may assume that $Q \leq U_1$, and set $K \defeq U_1$.
By our choice of $V$, we know that $Q$ is not a free factor of $U$, so $Q \lneq K$ 
and $K$ is infinite as $Q$ is maximal finite.
Applying the Kurosh subgroup theorem to the subgroup $R \defeq K \cap V$ of $V$, 
we see that $Q$ is a free factor of $R$, as required.
\end{proof}

We need the following more detailed formulation of \cite[Theorem 3.6]{WZ}.

\begin{thm} \label{VirtFreeProdThm}

Let $G$ be a finitely generated pro-$p$ group containing an open normal subgroup
$H$ that decomposes as a free pro-$p$ product 
\begin{equation}
H=\coprod_{i=1}^n H_i \amalg F
\end{equation}
where $H_i \ncong \Z_p$ are freely indecomposable pro-$p$ groups for $1 \leq i \leq n$,
the group $F$ is free pro-$p$,
and either $n \geq 2$ or $F$ is nontrivial.
Then there exists a finite connected graph of finitely generated pro-$p$ groups $(\mathcal{G},\Gamma)$ such that:

\begin{enumerate}

\item There exists a vertex $v_0 \in V(\Gamma)$ for which $G \cong \Pi_1(\mathcal{G},\Gamma,v_0)$. \label{cond1}
 
\item For every edge $e \in E(\Gamma)$ we have $H \cap \mathcal{G}(e) = \{1\}$. \label{cond2}

\item There exists a vertex $v \in V(\Gamma)$ such that $H_1 = H \cap \mathcal{G}(v)$.\label{cond3}
 
\end{enumerate}

\end{thm}

With \thmref{VirtFreeProdThm} at hand, we can finally prove the following.

\begin{prop} \label{OrderPProp}  

Let $G$ be a finitely generated freely indecomposable infinite pro-$p$ M. Hall group.
Then $G$ is torsion-free.

\end{prop}

\begin{proof}

We assume toward a contradiction that $G$ has torsion,
and use \lemref{maximal finite} to obtain a nontrivial maximal finite subgroup $Q$ of $G$.
As $G$ is M. Hall, there exists an open subgroup $H$ of $G$ such that $H = Q \amalg L$ for some $L \leq_c H$. 
By \propref{index p} and \propref{HallHeredProp} we may assume that $[G : H] = p$.

Take the smallest $\Gamma$ as in \thmref{VirtFreeProdThm},
and (using property \ref{cond3}) select a $v \in V(\Gamma)$ 
such that $Q = H \cap \mathcal{G}(v)$.
It follows that $\mathcal{G}(v)$ is finite, so by maximality, $Q = \mathcal{G}(v)$.
Choosing an $e \in E(\Gamma)$ incident to $v$ we see that 
\begin{equation}
\mathcal{G}(e) \leq \mathcal{G}(v) = Q \leq H
\end{equation}
but $H \cap \mathcal{G}(e) = \{1\}$ by property \ref{cond2}, so $\mathcal{G}(e) = \{1\}$.

If $e$ is a bridge, the minimality of $\Gamma$ implies that the fundamental group of each of the connected components created by removing $e$ is nontrivial, so we arrive at a contradiction to the indecomposability of $G$.
If $e$ is not a bridge, there exists a spanning tree of $\Gamma$ avoiding $e$,
so $e$ contributes a free $\mathbb{Z}_p$ factor to $G$, and we have the same contradiction once again.
\end{proof}

We are now in a position to prove \thmref{RibThm}.

\begin{proof} 

Let $G$ be a finitely generated freely indecomposable infinite pro-$p$ M. Hall group.  
Take $K$ to be a nontrivial procyclic subgroup of $G$ and let $U$ be an open subgroup of $G$ having $K$ as a free factor.
By \propref{OrderPProp}, $G$ is torsion-free, so $U = K$ by \cite[Corollary B]{WZ}, and $G$ is virtually free.
By a theorem of Serre from \cite{Ser0}, $G \cong \mathbb{Z}_p$ as required.
\end{proof}

\section*{Acknowledgments}
Mark Shusterman is grateful to the Azrieli Foundation for the award of an Azrieli Fellowship.
The first author was partially supported by a grant of the Israel Science Foundation with cooperation of UGC no. 40/14.
The second author was partially supported by CAPES and CNPq.
We would like to thank Lior Bary-Soroker, and Alexander Zalesskii for helpful discussions.
The second author is grateful to Lior Bary-Soroker for hospitality during
his visits to Tel-Aviv University; this project has started during these visits.

\end{document}